\documentclass[12pt]{amsart}
\usepackage[margin=2.23cm]{geometry}

\usepackage{amsthm}
\usepackage[normalem]{ulem}
\usepackage{hyperref}
\usepackage{amssymb,mathrsfs,mathcomp,amsfonts, textcomp, upgreek}
\usepackage[matrix,arrow,curve]{xy}
\usepackage{tabularx}

\newcommand{\mumu}{{\boldsymbol{\mu}}}

\newcommand{\CC}{\mathbb{C}}
\newcommand{\FF}{\mathbb{F}}
\newcommand{\QQ}{\mathbb{Q}}

\newcommand{\ZZ}{\mathbb{Z}}
\newcommand{\PP}{\mathbb{P}}
\newcommand{\OOO}{{\mathscr{O}}}

\newcommand{\NNN}{{\mathscr{N}}} 
\newcommand{\EEE}{{\mathscr{E}}}

\newcommand{\p}{\operatorname{p}_{\mathrm{a}}}
\newcommand{\g}{\operatorname{g}}
\newcommand{\red}{{\operatorname{red}}}
\newcommand{\qq}{\mathbin{\sim_{\scriptscriptstyle{\mathbb{Q}}} } }

\newcommand{\qQ}{\operatorname{q_{\QQ}}}

\newcommand{\Sing}{\operatorname{Sing}}
\newcommand{\Singng}{\operatorname{Sing}^{\mathrm{nG}}}
\newcommand{\mult}{\operatorname{mult}}

\newcommand{\Exc}{\operatorname{Exc}}
\newcommand{\Bs}{\operatorname{Bs}}

\newcommand{\Cl}{\operatorname{Cl}}
\newcommand{\rk}{\operatorname{rk}}
\newcommand{\J}{\operatorname{J}}
\newcommand{\aw}{\operatorname{aw}}
\newcommand{\ord}{\operatorname{ord}}

\newcommand{\typec}[1]{$(\mathrm{#1})$}
\newcommand{\typeci}[2]{$(\mathrm{#1}_{#2})$}
\newcommand{\type}[1]{$\mathrm{#1}$}
\newcommand{\typem}[1]{$\mathbf{#1}$}
\newcommand{\types}[2]{$\mathrm{#1}_{#2}$}
\newcommand{\typeA}[1]{$\mathrm{cA}_1^{#1}$}
\newcommand{\typeGor}[1]{$(\mathrm{NC}_{#1})$}

\newcommand{\xref}[1]{{\rm \ref{#1}}}
\makeatletter
\@addtoreset{equation}{subsection}
\makeatother

\theoremstyle{definition}

\swapnumbers
\theoremstyle{plain}
\newtheorem{theorem}[subsection]{Theorem}
\newtheorem{lemma}[subsection]{Lemma}
\newtheorem{proposition}[subsection]{Proposition}
\newtheorem{corollary}[subsection]{Corollary}
\newtheorem{claim}[subsection]{Claim}

\theoremstyle{definition}
\newtheorem{setup}[subsection]{Set-up}
\newtheorem{definition}[subsection]{Definition}

\newtheorem{assumption}[subsection]{Assumption}

\newtheorem{example}[subsection]{Example}
\newtheorem{remark}[subsection]{Remark}

\renewcommand{\theenumi}{\rm (\roman{enumi})}
\renewcommand{\labelenumi}{\rm (\roman{enumi})}

\begin{document}
\title{On the birational geometry of sextic threefold hypersurface in 
$\PP(1,1,2,2,3)$}

\address{ 
Steklov Mathematical Institute of Russian Academy of Sciences, Moscow, Russian 
Federation
} 
\email{prokhoro@mi-ras.ru}

\author{Yuri~Prokhorov}
\thanks{
This work was supported by the Russian Science Foundation under grant no. 23-11-00033, \url{https://rscf.ru/en/project/23-11-00033/}
}

\begin{abstract}
We investigate birational properties of hypersurfaces of degree $6$ in the 
weighted projective space 
$\PP(1,1,2,2,3)$. In particular, we prove that any such quasi-smooth hypersurface is not rational.
\end{abstract}

\maketitle

\section{Introduction}

We investigate birational properties of hypersurfaces of degree $6$ in the 
weighted projective space 
$\PP(1^2,2^2,3)$. Our interest to consider such hypersurfaces is motivated by 
the following result by
F.~Campana and H.~Flenner.

\begin{theorem}[{\cite{CampanaFlenner}}]
\label{thm:CaFl}
Let $Y$ be a three-dimensional projective variety containing 
a smooth rational surface 
$F$ as a 
Cartier divisor.
Assume that the normal bundle $\NNN_{F/Y}$ of $F$ in $X$ is ample. Then $Y$ is 
either rational or is birationally 
equivalent to a Fano variety $X$ with at worst terminal $\QQ$-factorial 
singularities, where $X$ 
belong to the following list:
\begin{enumerate}

\item
\label{thm:CaFl:3}
a smooth cubic hypersurface in $\PP^4$;

\item
\label{thm:CaFl:2}
a \emph{quartic double solid}, that is, 
a quartic hypersurface in $\PP(1^4, 2)$; 

\item
\label{thm:CaFl:1}
a \emph{double Veronese cone}, that is, a sextic hypersurface in $\PP(1^3, 2,3)$;

\item
\label{thm:CaFl:m}
a sextic hypersurface in $\PP(1^2,2^2,3)$.
\end{enumerate}
\end{theorem}
In the cases \ref{thm:CaFl:3}-\ref{thm:CaFl:1} the variety $X$ is a so-called 
\emph{del Pezzo threefold} (see e.g. \cite{Fujita:book}, \cite{KP:dP}).
It is known that such $X$ is not rational in the case \ref{thm:CaFl:3} as well 
as in the 
cases \ref{thm:CaFl:2} and \ref{thm:CaFl:1} under the assumption that $X$ is 
smooth
(see e.g. \cite{Clemens-Griffiths}, \cite{Voisin1988}, \cite{Grinenko:V1MFS}).
Rationality of \emph{singular} quartic double solids was studied in many papers (see \cite{Artin-Mumford-1972}, \cite{Beauville:Prym}, \cite{Debarre1990},
\cite{Prz-Ch-Shr:DS} and references therein). 
On the other hand, the birational geometry of varieties of type \ref{thm:CaFl:m}
is not studied well. 

Another motivation to consider sextic hypersurfaces in $\PP(1^2,2^2,3)$ is 
the study the birational geometry of
$\QQ$-Fano threefolds of large Fano index. Recall that a \emph{$\QQ$-Fano variety} is 
a projective variety $X$ with at worst terminal $\QQ$-factorial singularities 
such that the anticanonical class $-K_X$ is 
ample and the Picard rank $\uprho(X)$ equals $1$.
The \emph{$\QQ$-Fano index} of a $\QQ$-Fano variety is the 
maximal integer $q$ such that the relation 
$-K_X\qq q A$ holds for some integral Weil divisor $A$. 
We denote this number by $\qQ(X)$. According to \cite{Suzuki-2004} and 
\cite{P:2010:QFano} for a $\QQ$-Fano threefold $X$ we have
$\qQ(X)\in \{1,2,\dots,9,11,13,17,19\}$. Moreover, $X$ is rational if 
$\qQ(X)\ge 8$ \cite{P:2019:rat:Q-Fano}, \cite{P:QFano-rat1}.
It turns out that many $\QQ$-Fano threefolds of large $\QQ$-Fano index are
either rational or birationally equivalent to our hypersurface
$X_6\subset \PP(1^2,2^2,3)$ \cite{P:QFano-rat2}.

It is known that a \emph{very general} hypersurface
$X_6\subset \PP(1^2,2^2,3)$ not stably rational \cite{Okada2019}.
However, this strong result cannot be applied to a particular variety in the family. 
The goal of this note is to get rid of the generality condition and establish an explicit criteria 
for nonrationality.
Our main result is the following.
\begin{theorem}
\label{thm:main}
Let $X$ be a hypersurface of degree $6$ in the 
weighted projective space 
$\PP(1^2,2^2,3)$. Assume that the singularities of $X$ are terminal and $\QQ$-factorial.
\begin{enumerate}
\item \label{thm:main1}
If $X$ has a non-Gorenstein singularity which is not moderate \textup(see 
Definition~\xref{def:mode}\textup), then $X$ is rational.
\item \label{thm:main2}
If all non-Gorenstein singularities are moderate, any Gorenstein singularity of $X$ is a 
node or cusp, 
and the number of Gorenstein singularities is at most $4$, then $X$ is not rational.
\end{enumerate}
\end{theorem}

Note that a general hypersurface 
$X_6\subset \PP(1^2,2^2,3)$ satisfies the conditions of the lemma:
\begin{corollary}
Let $X$ be a quasi-smooth hypersurface of degree $6$ in the
$\PP(1^2,2^2,3)$. Then $X$ is not rational.
\end{corollary}

\subsection*{Acknowledgements}
The author would like to thank Alexander Kuznetsov and Ivan Cheltsov for helpful discussions.

\section{Preliminaries}
We work over the complex number field $\CC$.
\subsection{Notation} We employ the following standard notation.

\begin{itemize}
\item
$\PP(w_1,\dots,w_n)$ is the weighted projective space,
\item
$\FF_n$ is the Hirzebruch surface,
\item
$\mumu_n$ is the cyclic group of order $n$,
\item
$\Cl(X)$ is the Weil divisor class group of a normal variety $X$.
\end{itemize}

\subsection{Singularities}
For the classification of terminal threefold singularities we refer to 
\cite{Mori:term-sing} or \cite{Reid:YPG}.
Recall that any threefold terminal singularity $X\ni P$ of type \type{cA/r} is 
analytically isomorphic to the quotient
\begin{equation}
\label{eq:cA/r}
\{ xy+\phi(z^r,t)=0\} /\mumu_r(a,-a,1,0),\quad \gcd(r,a)=1.
\end{equation} 
The number $\aw(X,P):= \ord_0 \phi(0,t)$ is called the \emph{axial weight} of 
$X\ni P$. It coincides with the number of cyclic quotient singularities in a 
$\QQ$-smoothing 
\cite[\S~6.4]{Reid:YPG}.

\begin{definition}[\cite{Kawamata:Moderate}]
\label{def:mode}
We say that the singularity \eqref{eq:cA/r} is \emph{moderate} 
if the term $z^r$ appears in $\phi$.
Then by an analytic coordinate change the equation \eqref{eq:cA/r} can be rewritten 
as follows
\begin{equation}
\label{eq:cA/r:m}
\{ xy+z^r+t^m=0\} /\mumu_r(a,-a,1,0),\quad \gcd(r,a)=1,\quad m =\aw(X,P).
\end{equation} 
Here the case $m=1$ is not excluded. Then $X\ni P$ is a so-called \emph{terminal cyclic quotient 
singularity}, 
traditionally it is said to be of type $\frac 1r (a,-a,1)$ (see \cite{Reid:YPG}).
\end{definition}

Recall that an \emph{extremal blowup} of a threefold $X$ with terminal 
$\QQ$-factorial singularities
is a birational morphism $f: \tilde{X} \to X$ 
such that 
$\tilde{X}$ also has only terminal $\QQ$-factorial singularities and 
$\uprho(\tilde{X}/X)=1$.
In this situation the anticanonical divisor $-K_{\tilde X}$ must be $f$-ample. 

\begin{lemma}[see e.g. \cite{Kawamata-1992-e-app}]
\label{lemma:moderate-blowup}
Let $X\ni P$ be a moderate singularity given by \eqref{eq:cA/r:m} with 
$m>1$ and let $f: \tilde X\to X$ be the 
weighted blowup with weights 
\[\textstyle
\frac 1r(a,r-a,1,r). 
\]
Then $f$ is an extremal 
blowup, the exceptional divisor $E\subset \tilde X$ is Cartier, it
is isomorphic to the hypersurface 
\[
\{xy+z^r=0\}\subset \PP(a,r-a,1,r), 
\]
the 
discrepancy of $E$
equals $1/r$, and $\tilde X$ has \textup(at most\textup) three singular points 
which are cyclic quotients of types
\[\textstyle
\frac 1a(-r,r,1),\qquad \frac 1{r-a}(r,-r,1), 
\]
and a moderate singularity 
\[
\{ xy+z^r+t^{m-1}=0\} /\mumu_r(a,-a,1,0).
\]
\end{lemma}

\begin{corollary}[see e.g. {\cite[Ch.~10]{KM:92}}]
\label{cor:moderate-blowup}
In the above notation we have $E^3=r^2/ a(r-a)$.
\end{corollary}

\begin{corollary}
\label{cor:moderate-blowup1}
If in the above notation $r=2$, then the variety $\tilde X$ has a unique 
singularity, 
the divisor $E$ is Cartier, and $E\simeq \PP(1,1,4)$. 
\end{corollary}

\begin{definition}
A threefold singularity is said to be \emph{of type~\typeA{n}} if it
is analytically isomorphic to a hypersurface singularity given by the equation
\begin{equation*}
y_1^2+y_2^2+y_3^2+y_4^{n+1}=0.
\end{equation*} 
In particular, a singularity of type~\typeA{1} is a \emph{node} \textup(an ordinary double point\textup)
and a singularity of type~\typeA{2} is a \textup(generalized\textup) \textit{cusp}.
\end{definition}

\begin{remark}
\label{rem:cAn}
Let $X\ni P$ be a singularity of type \typeA{n}, let $f: \tilde X\to X$ be the blowup of the 
maximal ideal $\mathfrak{m}_{P,X}$, and let $E\subset \tilde X$ be the exceptional divisor. 
Then $f$ is an extremal blowup. Moreover, $E\simeq \PP(1^2, 2)$ if $n\ge 2$ and $E\simeq \PP^1\times \PP^1$ if $n=1$.
The variety $\tilde X$ has a unique singular point which is of type \typeA{n-2} if $n\ge 3$
and $\tilde X$ is smooth if $n=1$ or $2$.
\end{remark}

\section{Conic bundles}

\begin{definition}
A \emph{standard conic bundle} is a contraction $\pi:Y\to S$
from a smooth threefold to a surface such that all fibers are 
one-dimensional, $\uprho(Y/S)=1$, and $-K_Y$ is $\pi$-ample.
The \emph{discriminant divisor} of $\pi$ is the 
curve 
\[
\Delta_\pi:=\{s\in S \mid \text{ $\pi$ is not smooth over $s$}\}.
\]
\end{definition}

If $\pi:Y\to S$ is a standard conic bundle, then the base $S$ is smooth \cite{Mori:3-folds}. Moreover, 
$Y$ admits an embedding into $\PP^2$-bundle $\PP_{S}(\EEE)$ over $S$ so that the fibers of $\pi$ are plane 
conics, where $\EEE:=\pi_*\OOO_Y(-K_Y)$ is a rank-$3$ vector bundle (see e.g. \cite[Sect.~1]{Sarkisov:82e}).
The discriminant divisor $\Delta_\pi$ is a normal crossing curve (see \cite[Sect.~1]{Sarkisov:82e} or~\cite[Sect.~3]{P:rat-cb:e}).

Let $\hat \Delta_{\pi}$ 
is the scheme that parametrizes the components 
of degenerate fibers over $\Delta_{\pi}$. In other words, $\hat \Delta_{\pi}$ is the subscheme of the relative Grassmannian 
of lines in $\PP_S(\EEE)$ 
whose points correspond to lines lying in the fibers of $\pi$. Then the projection 
$\tau:\hat \Delta_{\pi}\to \Delta_{\pi}$ is a finite morphism of degree $2$.
We call this morphism the \emph{double cover associated to} $\pi$. 
If otherwise is not stated, the base surface $S$ of a conic bundle $\pi:Y\to S$ will be supposed to be projective
surface.

\begin{lemma}
\label{prop:cb}
Let $\pi: Y\to S$ be a standard conic bundle.
Suppose that a component $\Delta_1\subset \Delta_{\pi}$ of the discriminant curve is smooth. 
Then the number $\Delta_1\cdot (\Delta_{\pi}-\Delta_1)$ is even. 
If furthermore $\Delta_1\subset \Delta_{\pi}$ is rational, then 
$\Delta_1\cdot (\Delta_{\pi}-\Delta_1)\ge 2$.
\end{lemma}

\begin{proof}
Let $\tau:\hat \Delta_{\pi}\to \Delta_{\pi}$ be the double cover associated to $\pi$.
Its restriction $\tau_1:\hat \Delta_1\to \Delta_1$ over $\Delta_1$ is branched at
the points $\Delta_1\cap \overline{(\Delta_{\pi}\setminus \Delta_1)}$ and 
the number of this points must be even.
If $\Delta_1$ is a smooth rational curve and $\Delta_1\cap \overline{(\Delta_{\pi}\setminus \Delta_1)}=\varnothing$, then 
the 
double cover $\tau_1$ splits.
Hence the inverse image $\pi^{-1}(\Delta_1)$ is reducible.
This contradicts our assumption that $\pi$ is a
standard conic bundle.
\end{proof}

\subsection{Prym varieties}
Let $(\hat{\Gamma},\iota)$ be a pair consisting of a connected curve $\hat{\Gamma}$
with only nodes (ordinary double points) as singularities and an involution $\iota: \hat{\Gamma}\to \hat{\Gamma}$ 
satisfying the following condition: 
\begin{enumerate}
\item[(B)] 
The set of fixed points of $\iota$ 
coincides with the set of singular points of $\hat{\Gamma}$ and $\iota$ preserves the branches of $\hat{\Gamma}$ at all 
the singular points. 
\end{enumerate}
Such pairs $(\hat{\Gamma},\iota)$ are called \emph{Beauville pairs}.
For example, if $\tau:\hat \Delta_{\pi}\to \Delta_{\pi}$ is the double cover associated to a standard conic bundle $\pi: Y\to S$,
then there exists a natural involution $\iota:\hat \Delta_{\pi}\to\hat \Delta_{\pi}$
such that $\Delta_{\pi}=\hat \Delta_{\pi}/\langle\iota\rangle$ and $(\hat{\Gamma},\iota)$ is a Beauville pair.

To each Beauville pair $(\hat{\Gamma},\iota)$ 
one can associate a principally polarized 
abelian variety $\Pr (\hat{\Gamma},\iota)$ which is called the \emph{Prym variety} of $(\hat{\Gamma},\iota)$,
see \cite{Mumford1974}, \cite{Beauville:Prym}, and \cite{Shokurov:Prym}
for details. Abusing notation sometimes we will write $\Pr (\hat{\Gamma}/\Gamma)$ instead of $\Pr (\hat{\Gamma},\iota)$,
where $\Gamma:= \hat{\Gamma}/\iota$.

\begin{theorem}[{\cite{Shokurov:Prym}}]
\label{thm:sho}
In the above notation, suppose that the curve $\Gamma$ is connected, and the following condition
holds: 
\begin{enumerate}
\item[(S)] 
for any decomposition $\Gamma =\Gamma' + \Gamma''$ in a sum of effective 
divisors, one has $\# (\Gamma' \cap \Gamma'') \ge 4$. 
\end{enumerate}
Then $\Pr (\hat {\Gamma}/\Gamma)$ is
a sum of Jacobians of smooth curves if and only if $\Gamma$ is either 
hyperelliptic, or 
trigonal, or quasi-trigonal, or a
plane quintic curve such that 
the double 
cover $\hat {\Gamma}\to \Gamma$ corresponds to an even theta characteristic.
\end{theorem}

If $\p(\Gamma)\ge 2$, then the condition (S) implies that $\Gamma$ is a stable curve and the 
canonical map 
\[
\Phi_{|K_{\Gamma}|}:\Gamma \dashrightarrow \PP^{\p(\Gamma)-1}
\]
is a morphism that 
does not contract components. If the curve $\Gamma$ does not satisfy 
the condition (S), then in order to apply Theorem~\ref{thm:sho} one can decompose the Prym variety as follows.

\begin{remark}[{\cite[Corollary 3.16, Remark 3.17]{Shokurov:Prym}}]
\label{rem:Prym}
Let $(\hat{\Gamma},\iota)$ be a Beauville pair. Suppose that there exists a 
decomposition $\hat{\Gamma} = \hat{\Gamma}' \cup \hat{\Gamma}''$ with $\hat{\Gamma}' \cap \hat{\Gamma}''=\{P_1,\, P_2\}$. 
Denote by $\check{\Gamma}'$ and $\check{\Gamma}''$ the curves obtained
by identifying the points $P_1$ and $P_2$ by means of the involutions $\iota'$ and $\iota''$ induced by
$\iota$. Then $(\check{\Gamma}',\iota')$ and $(\check{\Gamma}'',\iota'')$ are also Beauville pairs and 
\[
\Pr (\hat {\Gamma}, \iota)= \Pr (\check{\Gamma}', \iota') \oplus \Pr (\check{\Gamma}'', \iota'').
\] 
In particular, if $\hat{\Gamma}''$ is a chain of rational curves, then
\[
\Pr (\hat {\Gamma}, \iota)= \Pr (\check{\Gamma}', \iota').
\]
\end{remark}

\begin{theorem}[{\cite{Clemens-Griffiths}}, {\cite{Beauville:Prym}}]
\label{thm:J-Pr}
Let $\pi:Y\to S$ be a standard conic bundle
and let $\tau:\hat \Delta_{\pi}\to \Delta_{\pi}$ is the double cover associated to $\pi$.
If $Y$ is rational, then the Prym variety $\Pr(\hat \Delta_{\pi}/\Delta_{\pi})$
is isomorphic 
\textup(as a principally polarized abelian variety\textup) to a sum of Jacobians of curves.
\end{theorem}

Note that the discriminant curve must be connected if $Y$ is rational \cite{Artin-Mumford-1972}.

\section{$\QQ$-conic bundles}
In this section we collect basic facts on three-dimensional Mori fiber spaces with 
one-dimensional 
fibers. For more detailed information and references
we refer to~\cite{MP:cb1}, \cite{MP:cb2}, \cite{P:ICM}.

\begin{definition}[\cite{MP:cb1}]
A \emph{$\QQ$-conic bundle} is a contraction $\pi:Y\to S$
from a threefold to a surface such that $Y$ is normal and has only terminal
singularities, all fibers are one-dimensional, and $-K_Y$ is $\pi$-ample.
We say that a $\QQ$-conic bundle $\pi:Y\to S$ is \emph{extremal} if $Y$ is 
$\QQ$-factorial and
the relative Picard number $\uprho(Y/S)$ equals $1$.
We say that $\pi:Y\to S$ is a \emph{$\QQ$-conic bundle germ} over a point 
$o\in S$
if $S$ (resp. $Y$) is regarded as a germ at $o$ (resp. along $\pi^{-1}(o)$).
The \emph{discriminant divisor} of a $\QQ$-conic bundle $\pi:Y\to S$ is the 
curve $\Delta_\pi\subset S$ that is the union of the one-dimensional
components of the set 
\[
\{s\in S \mid \text{ $\pi$ is not smooth over $s$}\}.
\]
\end{definition}

\begin{theorem}[{\cite[Theorem~1.2.7]{MP:cb1}}]
\label{thm:base-surface}
Let $\pi: Y\to S$ be a $\QQ$-conic bundle. Then the singularities of the base 
$S$ are at worst 
Du Val of type \type{A}.
\end{theorem}

\begin{corollary}
\label{cor:base-surface}
Let $\pi: Y\to S$ be a $\QQ$-conic bundle over a projective surface, where the variety $Y$ is 
rationally connected. Assume that $\Cl(Y)\simeq \ZZ\oplus\ZZ$. Then $\Cl(S)\simeq \ZZ$ and $S$ is 
isomorphic to one of the following four surfaces: 
\begin{enumerate}
\renewcommand{\theenumi}{\rm (\alph{enumi})}
\renewcommand{\labelenumi}{\rm (\alph{enumi})}
\item 
\label{cor:base-surface1}
$\PP^2$, 
\item 
\label{cor:base-surface2}
$\PP(1^2,2)$,
\item 
\label{cor:base-surface3}
$\PP(1,2,3)$, 
\item 
\label{cor:base-surface4}
a quasi-smooth hypersurface of degree $6$ in $\PP(1,2,3,5)$.
\end{enumerate}
Moreover, letting $h$ be the positive generator of $\Cl(S)$ we have 
$h^2=1$, $1/2$, $1/6$, and $1/5$ in the cases~\ref{cor:base-surface1},
\ref{cor:base-surface2},
\ref{cor:base-surface3},
\ref{cor:base-surface4}, respectively.
\end{corollary}

\begin{proof} 
Since both $Y$ and $S$ have only isolated singularities, the pull-back map $\pi^*: \Cl(S) \hookrightarrow 
\Cl(Y)$ is well-defined and it is injective. Therefore, $\Cl(S)\simeq \ZZ$. By Theorem~\ref{thm:base-surface}
the singularities of $S$ are at worst Du Val of type \type{A}. 
Since the surface $S$ is rationally connected, $-K_S$ is ample.
Then it follows from the main result of 
\cite{Miyanishi-Zhang:88} that the collection of singularities of $S$ is one of the following: 
$\varnothing$, \type{A_1}, \type{A_1}+\type{A_2}, \type{A_4}, and, moreover, 
the surface $S$ is determined uniquely up to isomorphisms by its singularity type \cite[Lemma~7]{Miyanishi-Zhang:88}.
On the other hand, the surfaces in~\ref{cor:base-surface1},
\ref{cor:base-surface2},
\ref{cor:base-surface3},
\ref{cor:base-surface4} are Du Val del Pezzos with desired types of singularities and $\Cl(S)\simeq \ZZ$.
\end{proof}

\begin{lemma}
\label{lemma:cb}
Let $\pi: Y\to S$ be a $\QQ$-conic bundle, where the variety $Y$ is projective, 
and let $\Delta_{\pi}\subset S$ be its discriminant curve.
Let
$h$ be an ample divisor on $S$ and let $F:=\pi^{-1}( h)$. Then
\begin{eqnarray}
\label{eq:cb:a}
K_Y\cdot F^2&=& -2 h^2,
\\
\label{eq:cb:b}
K_Y^2\cdot F&=&-4K_S\cdot h-h\cdot\Delta_{\pi}.
\end{eqnarray}
\end{lemma}

\begin{proof}
By the projection formula we have $K_Y\cdot F^2=K_Y\cdot \pi^* h^2=-2h^2$. 
Further,
we may assume that $h$ is very ample smooth curve and $F$ is a smooth surface. 
Then
\[
K_F^2=(K_Y+F)^2\cdot F=K_Y^2\cdot F+2K_Y\cdot F^2=K_Y^2\cdot F-4h^2.
\]
On the other hand, by Noether's formula
\[
K_F^2=12\chi(\OOO_F)-\upchi_{\mathrm{top}}(F)= 12(1-\g(h))- 
(2-4\g(h)+2+h\cdot\Delta_{\pi})=
-4(K_S+h)\cdot h-h\cdot\Delta_{\pi}.
\]
Combining, we obtain
\[
K_Y^2\cdot F-4h^2=K_F^2=-4K_S\cdot h-4h^2-h\cdot\Delta_{\pi},
\qquad 
K_Y^2\cdot F+h\cdot\Delta_{\pi}=-4K_S\cdot h.\qedhere
\]
\end{proof}

Below we provide several ``basic'' examples of $\QQ$-conic bundles.
We are interested in the local structure of them near the singular fiber.
Denote by $\zeta_r$ a primitive $r$-th root of unity.

\begin{example}[Type \typec{IF_1}]
\label{ex:ODP}
The variety $Y$ is given in $\PP^1_{y_1,y_2, y_3}\times
\CC^2_{u,v}$ by the equation 
\[
y_1^2+y_2^2+uvy_3^2=0,
\]
and $\pi\colon Y\to S$ is the
projection to 
$S:=\CC^2$.
The singular locus of $Y$ consists of one node, 
$\Delta_\pi=\{uv=0\}$, and the central fiber $C:=\pi^{-1}(0)_{\mathrm{red}}$ 
is a pair of lines.
\end{example}

\begin{example}[Type \typeci{T}{r}]
\label{ex:T}
The variety $Y$ is the quotient of $\PP^1_{y_1, y_2}\times
\CC^2_{u,v}$ by the $\mumu_r$-action
\[
(y_1, y_2;\, u,v) \longmapsto(y_1,\zeta_r\, y_2;\, \zeta_r^a\, u,\zeta_r^{-a}\, 
v),
\]
where $\gcd (r,a)=1$, and $\pi\colon Y\to S$ is the
projection to 
$S:=\CC^2/\mumu_r$. The singular locus of $Y$ consists of two (terminal) cyclic
quotient singularities of types $\frac1r(1,a,-a)$ and
$\frac1r(-1,a,-a)$. The singularity of $S$ at the origin is of type 
\types{A}{r-1}.
In this case $\Delta_\pi=\varnothing$ and for the central fiber 
$C:=\pi^{-1}(0)_{\mathrm{red}}$ we have $C\simeq\PP^1$ and
$-K_Y\cdot C=2/r$.
\end{example}

\begin{example}[Type \typec{k2A}]
\label{ex:k2A}
The variety $Y$ is the quotient of the hypersurface
\[
Y'= \{ y_1^2+uy_2^2+vy_3^2=0\}\subset \PP^2_{y_1,y_2,y_3} \times\CC^2_{u,v}
\]
by the $\mumu_{r}$-action
\[
(y_1,y_2,y_3;\, u,v)\longmapsto (\zeta_r^{a}\, y_1,\zeta_r^{-1}\, y_2,y_3;\, 
\zeta_r\, u,\zeta_r^{-1}\, v),
\]
where $r=2a+1$, and $\pi\colon Y\to S$ is the
projection to 
$S:=\CC^2/\mumu_r$. 
The singular locus of $Y$ consists of two (terminal) cyclic quotient
singularities of types $\frac 1r(a,-1,1)$ and $\frac 1r(a+1,1,-1)$.
The singularity of $S$ at the origin is of type \types{A}{r-1}.
In this case $\Delta_\pi=\{uv=0\}/\mumu_r$ and for the central fiber 
$C:=\pi^{-1}(0)_{\mathrm{red}}$ we have $C\simeq\PP^1$ and
$-K_Y\cdot C=1/r$.
\end{example}

\begin{example}[Type \typec{ID_1^\vee}]
\label{ex:ID}
The variety $Y$ is the quotient of the hypersurface
\[
\{y_1^2+y_2^2+uv y_3^2=0\}\subset \PP^2_{y_1,y_2,y_3}\times
\CC^2_{u,v}\},
\]
by the $\mumu_2$-action
\[
(y_1,y_2,y_3;\, u,v)\longmapsto (-y_1,y_2,y_3;\, -u,-v).
\]
Then $Y$ has a unique singular point which is moderate of axial weight $2$.
The singularity of $S$ at the origin is of type \types{A}1.
In this case $\Delta_\pi=\{uv=0\}/\mumu_m$ and for the central fiber 
$C:=\pi^{-1}(0)_{\mathrm{red}}$ we have $C\simeq\PP^1$ and
$-K_Y\cdot C=1$.
\end{example}

It is easy to see that in all cases \typec{IF_1}, \typeci{T}{r}, \typec{k2A} 
and \typec{ID_1^\vee}
the pair $(S,\Delta_{\pi})$ is lc at the origin. The converse is almost true: 

\begin{proposition}[{\cite[Corollary~11.8.1]{P:rat-cb:e}}]
\label{prop-rem:lc}
Suppose that $\pi:Y\to S\ni o$ is a $\QQ$-conic bundle germ
such that the pair $(S,\Delta_{\pi})$ is lc and $Y$ is not smooth along 
$\pi^{-1}(o)$. Then $\pi$ is biholomorphically equivalent to a $\QQ$-conic bundle germ 
described in Examples \xref{ex:ODP}, \xref{ex:T}, \xref{ex:k2A}, or~\xref{ex:ID}.
\end{proposition}

We also need to consider Gorenstein $\QQ$-conic bundles 
whose total space has singularities nodes or cusps that is more general than \typec{IF_1}:

\begin{definition}
We say that a Gorenstein $\QQ$-conic bundle germ $\pi:Y\to S\ni o$ is of \emph{type \typeGor{n}} if 
the singularities of $Y$ are at worst nodes or cusps and 
the number of singular points equals $n$.
\end{definition}

\begin{lemma}
\label{lemma:new}
Let $\pi:Y\to S\ni o$ be a $\QQ$-conic bundle germ of type \typeGor{n}. 
Then the following assertions hold:
\begin{enumerate}

\item
$n\le 2$,

\item
if $o\notin \Sing(\Delta_{\pi})$, then $Y$ is smooth along $\pi^{-1}(o)$; 
if $o\in \Sing(\Delta_{\pi})$, then 
$o\in \Delta_{\pi}$ is a double point.

\end{enumerate}
\end{lemma}
\begin{proof}
The assertions are trivial if $Y$ is smooth, so we assume that $Y$ is singular.
Since the singularities of $Y$ are Gorenstein terminal, the base $S$ is smooth (see e.g. \cite{Cutkosky:contr} or \cite[Theorem~10.2]{P:G-MMP}).
There exists an embedding
$Y\subset \PP^2 \times S$ such that the equation of $Y$ can be written as 
follows:
\begin{equation}
\label{eq:eq:conic}
\sum_{0\le i,\, j\le 2} \phi_{i,j}(u,v) x_ix_j=0,\quad \phi_{i,j}(u,v)\in 
\OOO_{o,S},\quad \phi_{i,j}(u,v)=\phi_{j,i}(v,u).
\end{equation}
The fiber $Y_o$ is a degenerate conic 
and so $1\le \rk \|\phi_{i,j}(0,0)\|\le 2$.
Thus we have the following possibilities.

\subsubsection*{Case: $\rk \|\phi_{i,j}(0,0)\|=2$.}
Then by a coordinate change, which is linear in $x_0, x_1, x_2$, we can reduce 
\eqref{eq:eq:conic} to a diagonal form 
\[
x_0^2+x_1^2+\phi(u,v) x_2^2=0
\]
and then the equation of $\Delta_{\pi}$ is $\phi=0$. 
The point $P:=(0,0,1;0,0)$ is a unique singularity of~$Y$, hence $n=1$.
Since $P\in Y$ is of type~\typeA{m}, we have $\mult_0(\phi)\le 2$. 

\subsubsection*{Case: $\rk \|\phi_{i,j}(0,0)\|=1$.}
As above we can reduce \eqref{eq:eq:conic} to the following form
\begin{equation*}
x_0^2+\phi_{1,1}(u,v) x_1^2+2\phi_{1,2}(u,v) x_1x_2+\phi_{2,2}(u,v)x_2^2=0, 
\end{equation*}
where $\phi_{1,1}(0,0)=\phi_{1,2}(0,0)=\phi_{2,2}(0,0)=0$. 
The singularities of $Y$ on $\pi^{-1}(o)$
are given by 
\[\textstyle
x_0=u=v=x_1^2\frac{\partial \phi_{1,1}}{\partial u} +2x_1x_2\frac{\partial \phi_{1,2}}{\partial u}+
x_2^2\frac{\partial \phi_{2,2}}{\partial u}
= x_1^2\frac{\partial\phi_{1,1}}{\partial u} +2x_1x_2\frac{\partial \phi_{1,2}}{\partial u}
+x_2^2\frac{\partial \phi_{2,2}}{\partial v}=0.
\]
Since the singularities of $Y$ are isolated, this implies that $n\le 2$.

The discriminant curve $\Delta_{\pi}$ is given by the equation $\phi_{1,2}^2-\phi_{1,1}\phi_{2,2}=0$.
Assume that $\mult_0 (\phi_{1,2}^2-\phi_{1,1}\phi_{2,2})> 2$. 
By a linear coordinate change we also may assume that the point $P:=(0,0,1;0,0)$ is singular and its local equation is 
\[
x_0^2+\phi_{1,1}(u,v) x_1^2+2\phi_{1,2}(u,v) x_1+\phi_{2,2}(u,v)=0,
\]
where $\mult_0 \phi_{2,2}\ge 2$.
Then $\mult_0 (\phi_{1,2})\ge 2$, hence the rank of the quadratic part
of $\phi_{2,2}$ equals~$2$. Then by an analytic coordinate change of $u$ and $v$ we may assume that $\phi_{2,2}=u^2+v^2$.
Since $\phi_{1,1}(0)=0$ and $\mult_0 (\phi_{1,2})\ge 2$, we see that $P\in Y$ is of type~\typeA{n} with $n\ge 3$, a contradiction.
\end{proof}

The following fact is a direct consequence of \cite[Theorem~1.2]{MP:cb1} and 
\cite[Theorem~1.3]{MP:cb2}.
\begin{proposition}
\label{prop:cb-index2}
Let $\pi:Y\to S\ni o$ be a $\QQ$-conic bundle germ, where $S\ni o$ is a 
singularity of type \type{A_1}.
If the singularities of $Y$ are moderate of indices $\le 2$, then $\pi$ is biholomorphically equivalent to a $\QQ$-conic bundle germ 
described in Examples \xref{ex:T} or \xref{ex:ID}. 
\end{proposition}

\begin{theorem}[cf. {\cite[Sect. 11]{P:rat-cb:e}}]
\label{thm:SL-CB}
Let $\pi: Y\to S\ni o$ be a $\QQ$-conic bundle germ of one of the types \typeGor{n}, 
\typeci{T}{r}, \typeci{k2A}{r}, or \typec{ID_1^\vee}.
Let $P\in Y$ be a singular point, let $r$ be its index, and let $p: \check Y\to Y$ be 
an extremal blowup 
of $P$ with discrepancy $1/r$.
Then $p$ can be completed to a type \typem{I} Sarkisov link 
\begin{equation}
\label{eq:SL-CB}
\vcenter{
\xymatrix@R=1em{
\check{Y}\ar[d]_{p}\ar@{-->}[r]^{\chi}& Y'\ar[d]^{\pi'}
\\
Y\ar[d]_{\pi} & S'\ar[ld]_{\alpha}
\\
S
} }
\end{equation}
where $\chi$ is a birational transformation that is isomorphism in codimension 
one, 
$\pi'$ is a $\QQ$-conic bundle, and $\alpha$ is a crepant contraction with 
$\uprho(S'/S)=1$ in the cases \typec{T}, \typec{k2A}, 
\typec{ID_1^\vee},
and $\alpha$ is the blowup of $o$ in the case \typeGor{n}.
\begin{enumerate}
\item 
\label{thm:SL-CB:o}
If $\pi: Y\to S\ni o$ is of type \typeGor{1}, then $Y'$ is smooth
and $\Delta_{\pi'}$ is the proper transform of $\Delta_{\pi}$.

\item 
\label{thm:SL-CB:o1}
If $\pi: Y\to S\ni o$ is of type \typeGor{2}, then $\pi': 
Y'\to S'$ is a Gorenstein conic bundle whose unique singular fiber is of type \typeGor{1}
and $\Delta_{\pi'}$ is the proper transform of $\Delta_{\pi}$.

\item 
\label{thm:SL-CB:a}
If $\pi: Y\to S\ni o$ is of type \typeci{T}{r}, then $\pi': Y'\to S'$ has two 
\textup(or 
one\textup) singular fibers
and corresponding $\QQ$-conic bundle germs are of type \typeci{T}{a} and 
\typeci{T}{r-a}.
\item 
\label{thm:SL-CB:b}
If $\pi: Y\to S\ni o$ is of type \typeci{k2A}{r}, then $\pi': Y'\to S'$ has two 
\textup(or 
one\textup) singular fibers
and corresponding $\QQ$-conic bundle germs are of type \typeci{k2A}{r-2} and 
of type \typec{ID_1^\vee}.
In this case $\Delta_{\pi'}=\alpha^*\Delta_{\pi}$.
\item 
\label{thm:SL-CB:c}
If $\pi: Y\to S\ni o$ is of type \typec{ID_1^\vee}, then $Y'$
is smooth and $\Delta_{\pi'}=\alpha^*\Delta_{\pi}$.
\end{enumerate}
In all cases we have
\begin{equation}
\label{eq:log-crepant}
\textstyle
K_{S'}+\frac12 \Delta_{\pi'} = \alpha^*\left(K_S+ \frac12 \Delta_{\pi}\right).
\end{equation} 
\end{theorem}

\begin{proof}
Apply \cite[Construction~11.1]{P:rat-cb:e}.
In the cases \typeci{T}{r} and \typeci{k2A}{r} the assertion follows from 
Examples~11.5 and~11.6 in \cite{P:rat-cb:e}.
The case \typeGor{n} is similar to \cite[Lemma~8]{Avilov:cb} (see also 
\cite[Example~11.8]{P:rat-cb:e}). In this case $p$ is the blowup of the maximal ideal of 
the \typeA{n_1}-point (see Remark~\ref{rem:cAn}), $\chi$ is a flop, and $\Delta_{\pi'}$ does not contain 
$\alpha$-exceptional divisor by \cite[Lemma~10.10]{P:rat-cb:e}.

Consider the case \typec{ID_1^\vee}. Then $P\in Y$ is the only singularity of 
$Y$
and it is moderate with $\aw(Y,P)=2$. By Lemma~\ref{lemma:moderate-blowup} the 
variety
$\check{Y}$ has a unique singularity, say $Q$, which is of type $\frac 12(1,1,1)$.
Let $E$ be the $p$-exceptional divisor, let $C:=\pi^{-1}(o)_{\red}$ be the 
central fiber and let $\check{C}\subset \check{Y}$ be its proper transform.
Local computations show that $\check{C}$ passes through $Q$ and $E\cdot \check{C}=1$.
Then as in \cite[(11.1.2)]{P:rat-cb:e} we have 
\[\textstyle
K_{\check{Y}}\cdot \check{C}=K_Y\cdot C+\frac 
12 
E\cdot \check{C}=-\frac 12.
\]
Hence $\check{C}$ is a flipping curve and the link \eqref{eq:SL-CB} exists by 
\cite[Construction~11.1]{P:rat-cb:e}.
Here $\alpha$ is a crepant contraction and so $S'$ is smooth.
Moreover, by \cite[Theorem~4.2]{KM:92} the variety $Y'$ is smooth and $\chi: 
\check{Y}\dashrightarrow Y'$ is a single flip.
It remains to show that $\Delta_{\pi'}=\alpha^*\Delta_{\pi}$.
Let $\Delta_1$, $\Delta_2$ be analytic 
components of $\Delta_{\pi}$
and let $\Delta_1'$, $\Delta_2'$ be their proper transforms on $S'$.
These curve are smooth and $\Delta_1'\cap\Delta_2'=\varnothing$ (because the 
pair $(S,\Delta_\pi)$ is lc).
Moreover, the inverse images $\pi^{-1}(\Delta_i)$ are irreducible.
Hence the same holds for $\pi'^{-1}(\Delta_i')$. This implies that the fibers 
over the intersection points 
$\Delta_i'\cap\alpha^{-1}(o)$ are double lines and so $\alpha^{-1}(o)\subset 
\Delta_{\pi'}$
(see e.g. \cite[\S~3.3]{P:rat-cb:e}).
\end{proof}

\section{Hypersurface $X_6\subset \PP(1^2,2^2,3)$}

\begin{assumption}
\label{ass:3}
Throughout this section $X$ denotes
a hypersurface of degree $6$ in $\PP(1^2,2^2,3)$.
We assume that the singularities of $X$ 
are terminal.
\end{assumption}

\begin{proposition}
\label{setup:q=3}
In some coordinate system $x_1,y_1,x_2,y_2,x_3$ the equation of $X$ has
one of the following forms:
\begin{align}
\label{eq:q=3o:eq3}
&& x_3^2+x_2y_2(x_2+y_2)+
x_2^2\phi_2 +
x_2y_2\phi_2'+
y_2^2\phi_2''+
x_2\phi_4+
y_2\phi_4'+\phi_6=0,
\\
\label{eq:q=3o:eq2}
&& x_3^2+x_2^2y_2+
y_2^2\phi_2''+
x_2\phi_4+
y_2\phi_4'+\phi_6=0,
\\
\label{eq:q=3o:eq1}
&& x_3^2+x_2^3+
x_2y_2\phi_2'+
y_2^2\phi_2''+
x_2\phi_4+
y_2\phi_4'+\phi_6=0,
\end{align}
where
$\phi_d$, $\phi_d'$, $\phi_d''$
are homogeneous polynomials in $x_1,y_1$ and the subscript is the degree.
Moreover, these polynomials and the set $\Singng(X)$ of non-Gorenstein points 
satisfy the following conditions:
\begin{center}
\begin{tabularx}{0.99\textwidth}{l|l|X}
equation&conditions&$\Singng(X)$
\\\hline
\eqref{eq:q=3o:eq3} & & $3\times \frac12(1,1,1)$
\\
\eqref{eq:q=3o:eq2}& $\phi_2''\neq 0$ & $\frac12(1,1,1)$ $+$ \type{cA/2},\quad 
moderate with $\aw=2$ iff $\rk\phi_2''=2$
\\
\eqref{eq:q=3o:eq2}&$\phi_2''=0$ & $\frac12(1,1,1)$ $+$ \type{cAx/2}
\\
\eqref{eq:q=3o:eq1}&$\phi_2''\neq 0$ & \type{cA/2},\quad moderate with $\aw=3$ 
iff 
$\rk\phi_2''=2$ 
\\
\eqref{eq:q=3o:eq1}&$\phi_2''=0$, $\phi_2'\neq 0$ & \type{cD/2}
\\
\eqref{eq:q=3o:eq1}&$\phi_2''=0$, $\phi_2'= 0$, $\phi_4'\neq 0$ & \type{cE/2}
\\\hline
\end{tabularx} 
\end{center}
\end{proposition}

\begin{proof}
Let $\phi=0$ be an equation of $X$. If $\phi$ does not contain $x_3^2$, then 
the 
point $(0,0,0,0,1)$
is the quotient of a hypersurface singularity by $\mumu_3(1,1,2,2)$. Such a 
point cannot be terminal.
Thus $\phi\ni x_3^2$. Completing the square we may assume that $\phi$ does not 
contain other terms that depend on $x_3$.
Then the set of non-Gorenstein points of $X$ is given by $\phi=x_1=y_1=x_3=0$. 
Hence the cubic form $\phi(0,0,x_2,y_2,0)$
is not identically zero. It may have either three, two or one distinct linear 
factors.
These cases correspond to the equations \eqref{eq:q=3o:eq3}, \eqref{eq:q=3o:eq2}, and 
\eqref{eq:q=3o:eq1}, respectively.
The information in the table follows from the classification of terminal 
singularities, see \cite{Mori:term-sing}, \cite{Reid:YPG},
and \ref{def:mode}.
\end{proof}

Recall that 
the $\QQ$-Fano index of a singular Fano variety is the 
maximal integer $q=\qQ(X)$ such that the relation 
$-K_X\qq q A$ holds for some integral Weil divisor $A$. 
The divisor $A_X$ satisfying the relation $-K_X\qq \qQ(X) A_X$
is called the \emph{fundamental divisor}.

\begin{corollary}
\label{cor:X6}
In the above notation, the following assertions hold:
\begin{enumerate}
\item \label{cor:X6:1}
the singularity basket of $X$ \textup(see \cite{Reid:YPG}\textup) consists of three points of type 
$\frac12(1,1,1)$;
\item \label{cor:X6:2}
$\OOO_X(A_X)=\OOO_X(1)$;
\item \label{cor:X6:3}
$\qQ(X)=3$;
\item \label{cor:X6:4}
$A_X^3=1/2$;
\item 
\label{cor:X6:5}
the group $\Cl(X)$ is torsion free \textup(see \cite[Proposition~2.9]{P:2010:QFano}\textup).
\end{enumerate}
\end{corollary}

\begin{remark}
According to \cite[Theorem~1.5]{CampanaFlenner} or \cite{Sano-1996} 
the properties~\ref{cor:X6:1}, \ref{cor:X6:3}, and~\ref{cor:X6:4} characterize hypersurfaces $X_6\subset \PP(1^2,2^2,3)$: any Fano threefold with
terminal singularities of Gorenstein index $\le 2$ such that $\qQ(X)=3$
and $A_X^3=1/2$
is isomorphic to a hypersurface $X_6\subset \PP(1^2,2^2,3)$.
\end{remark}

\begin{corollary}
\label{lemma:q=3o:pp}
In the notation of \xref{ass:3}
the linear system 
$|2A_X|$ is base point free and defines a double cover $X\to Q\subset \PP^4$, 
where 
$Q$ is a quadric of corank $2$, i.e. $Q\simeq \PP(1^2,2^2)$. 
The branch divisor is an element of the anticanonical linear system $|-K_Q|$.
\end{corollary}

\begin{proof}
The map $\Phi_{|2A_X|}: X \dashrightarrow \PP^4$ is given by $(x_1,y_1,x_2,y_2, 
x_3) \longmapsto (x_1^2, x_1y_1,y_1^2,x_2,y_2)$.
It is easy to see that this map is a morphism and its image is the quadric 
$\{z_1z_3-z_2^2=0\}$.
\end{proof}

\begin{corollary}
\label{cor:n-mod-rat}
If in the notation of \xref{ass:3} the variety $X$ contains a non-Gorenstein 
singularity that is worse than 
moderate, then $X$ is rational.
\end{corollary}

\begin{proof}
Indeed, if $X$ has a singularity that is worse than of type \type{cA/2}, then 
$X$ 
is given by an equation of the form \eqref{eq:q=3o:eq2} or \eqref{eq:q=3o:eq1}
which does not contain $y_2^2$ but contains ether $x_2^2y_2$, $x_2y_2\phi_2'$ 
or 
$y_2\phi_4'$, hence it is rational.
If $X$ has a singularity of type \type{cA/2} that is not moderate, then it 
is given by an equation of the form \eqref{eq:q=3o:eq2} or \eqref{eq:q=3o:eq1} 
with $\rk \phi_2''=1$.
Then we may assume that $\phi_2''=-x_1^2$ and so in the affine chart $x_1\neq 
0$ 
the equation of $X$ has the form
\begin{align*}
\tag{\ref{eq:q=3o:eq2}${}^\prime$}
(x_3-y_2)(x_3+y_2)+x_2^2y_2+x_2\phi_4+\phi_4'+\phi_6=0,
\\
\tag{\ref{eq:q=3o:eq1}${}^\prime$}
(x_3-y_2)(x_3+y_2)+x_2^3+
x_2y_2\phi_2'+x_2\phi_4+
\phi_4'+\phi_6=0.
\end{align*}
Then applying the coordinate change $x_3'=x_3-y_2$, $y_2'=x_3+y_2$ we can see 
that $X$ is rational.
\end{proof}

\begin{setup}
\label{setup:q=3a}
From now on we assume that $X$ is a $\QQ$-Fano variety and any non-Gorenstein 
singularity of $X$ is moderate. 

\end{setup}

\begin{theorem}
\label{thm:q=3o:main}
In the above notation there exist 
the following Sarkisov link
\begin{equation}
\label{eq:q=3o:sl}
\vcenter{
\xymatrix@C=3em{
& \tilde X\ar@/_0.7em/[dl]_{f}\ar@/^0.7em/[dr]^{\varphi}
\\
X\ar@{-->}[rr]&&\PP(1^2,2)
} }
\end{equation}
where $f$ is an extremal blowup of a point of index $2$, as in Lemma~\xref{lemma:moderate-blowup}, and
$\varphi$ is an extremal $\QQ$-conic bundle.
The discriminant curve $\Delta_{\varphi}\subset\PP(1^2,2)$ is a member of the 
linear system $|-2K_{\PP(1^2,2)}|$.
The germ of $\varphi$ over the point $o:=(0,0,1)\in \PP(1^2,2)$
is of type \typeci{T}{2} or \typec{ID_1^\vee} and all non-Gorenstein points of 
$\tilde X$ are contained in the fiber 
$\varphi^{-1}(o)$.
\end{theorem}

The rest of this section is devoted to the proof of this theorem.
We will assume that $X$ is given by one of the 
equations~\eqref{eq:q=3o:eq3}-\eqref{eq:q=3o:eq1}.
Let $P\in X$ be a point of index $2$, let
\[
f: \tilde X\longrightarrow X
\]
be an extremal blowup of $P$ (see
Lemma~\xref{lemma:moderate-blowup}), and let $E$ be the $f$-exceptional divisor.

\begin{claim}
\label{claim:Sing-tildeX}
The set of non-Gorenstein singularities of $\tilde X$ consists of either 
two cyclic quotients $\frac 12(1,1,1)$ or a single moderate singularity of axial weight $2$.
\end{claim}

\begin{proof}
Follows from Proposition~\ref {setup:q=3} and Lemma~\ref{lemma:moderate-blowup}.
\end{proof}

Since $\Cl(X)$ is torsion free, so is the group $\Cl(\tilde X)$. Then
$\Cl(\tilde X)\simeq \ZZ\oplus \ZZ$.
Recall that 
the discrepancy of the $f$-exceptional divisor $E$ equals $1/2$ (see Lemma~\ref{lemma:moderate-blowup}). 
Hence,
\[\textstyle
K_{\tilde X}\qq f^*K_X+\frac12 E. 
\]
Let $M_1\in |A_X|$ and $M_2\in|2A_X-P| $ be general members, where $|2A_X-P|\subset |2A_X|$
is the linear subsystem consisting of all divisors passing through $P$.
By $\tilde M_1$ and $\tilde M_2$
we denote their proper transforms on $\tilde X$.
Since $f$ is the weighted blowup as in Lemma~\xref{lemma:moderate-blowup}, the following two relations can be checked locally by using the 
equations~\eqref 
{eq:q=3o:eq3}-\eqref{eq:q=3o:eq1}:
\begin{align*}
\tilde{M}_1&\qq f^*M_1 -\textstyle \frac12 E,
\\
\tilde{M}_2&\sim f^*M_2 - E.
\end{align*}
Note that the divisors $\tilde{M}_2$ and $K_{\tilde X}+\tilde{M}_1$ are 
Cartier. Further,
\begin{equation}
\label{eq:q=3:M2M1}
\tilde{M}_2\sim 2\tilde{M}_1\sim f^*(2A_X)-E,
\qquad 
K_{\tilde X}+\tilde{M}_1\sim -f^*(2A_X).
\end{equation}
We also have
\[
f^*A_X\cdot E^2=(f^*A_X)^2\cdot E=0, \qquad E^3=4
\]
(see Corollary~\ref{cor:moderate-blowup}).

\begin{claim}
\label{lemma:q=3o:BL}
\begin{enumerate}
\item 
\label{lemma:q=3o:BL1}
The divisors $\tilde{M}_2$ and $\tilde{M}_1$ are nef.
\item 
\label{lemma:q=3o:BL3}
The divisor $-K_{\tilde X}$ is ample.
\end{enumerate}
\end{claim}

\begin{proof}
For the proof of \ref{lemma:q=3o:BL1},
assume that $\tilde{M}_2\cdot \tilde C< 0$ for some irreducible curve $\tilde 
C\subset \tilde X$. 
We have $\tilde{M}_2\sim f^* (2A_X)-E$.
Since $\uprho(\tilde X/X)=1$, the restriction of $\tilde{M}_2$ to $E$ is ample.
Therefore, $\tilde C\not\subset E$ and the curve $C:=f(\tilde C)$ is contained 
in the base locus of the linear system $|2A_X-P|$.
On the other hand, $\Bs |2A_X-P|$ is zero-dimensional.
Indeed, if $P$ is the point $(0,0,1,0,0)$, then 
$\Bs |2A_X-P|\subset \{x_1=y_1=x_2=x_3=0\}$
(for other choices of $P$ the arguments are the same). 
Hence $\Bs |2A_X-P|$ cannot contain a curve, a contradiction.
Thus $\tilde{M}_2$ is nef and so $\tilde{M}_1$ is by \eqref{eq:q=3:M2M1}. 
This proves \ref{lemma:q=3o:BL1}. For the proof of \ref{lemma:q=3o:BL3} we just 
note that there is a decomposition
$-K_{\tilde X}\sim \tilde{M}_1+f^*(2A_X)$, where the summands are nef and not 
proportional.
\end{proof}

Since $-K_{\tilde X}$ is ample and $\uprho(\tilde X)=2$, there exists an 
extremal Mori contraction $\varphi:\tilde X\to S$ other than $f$. 
Since the divisor $\tilde{M}_1$ is nef and
$\tilde{M}_1^3=(f^*A_X -\frac 12 E)^3=A_X^3-\frac18 E^3=0$, we see that $\tilde{M}_1$ is
trivial on the fibers of $\varphi$ and the contraction $\varphi$ is not 
birational.
Since 
\[\textstyle
\tilde{M}_1^2\cdot (-K_{\tilde X})=\left(f^*A_X -\frac 12 E\right)^2 \cdot 
\left(3f^*A_X -\frac 12 E\right)=3A_X^3-\frac18 E^3=1,
\]
the target $S$ of the contraction is not a curve. Hence
$\varphi$ is a $\QQ$-conic bundle.
We can write $\tilde{M}_1\sim \varphi^*(kh)$, where $h$ is a positive generator of $\Cl(S)\simeq \ZZ$
and $k$ is a positive integer. Then $k^2h^2=1/2$ by~\eqref{eq:cb:a}.
Then by Corollary~\ref{cor:base-surface} we have $k=1$, $h^2=1/2$,
and $S\simeq 
\PP(1^2,2)$.
By Proposition~\ref{prop:cb-index2} the germ over $o\in S$ is of type 
\typec{T} or \typec{ID_1^\vee}. Then all the non-Gorenstein singularities of 
$\tilde X$ lie on 
$\varphi^{-1}(o)$ (see Claim~\ref{claim:Sing-tildeX}).
Finally, $(-K_{\tilde X})^2\cdot \tilde M_1=4$, hence
by \eqref{eq:cb:b} we have $\Delta_{\pi}\in |-2K_{S}|$.

\begin{corollary}[{\cite[Proposition~4.15]{CampanaFlenner}}, 
{\cite[Proposition~6.4]{P:QFano-rat1}}]
In the assumptions of Theorem~\xref{thm:q=3o:main} the variety $X$ is 
unirational of degree at most~$2$.
\end{corollary}

\begin{proof}
The divisor $E$ is rational and the restriction $\varphi: E\to S$ 
is finite of degree $2$. Then the base change $\tilde X\times_S E$ 
is rational and 
the induced morphism $\tilde X\times_S E\to X$ is finite of degree~$2$.
\end{proof}

\begin{remark}
Assume that in coordinates $x_1,y_1,x_2,y_2, x_3$ in $\PP(1^2,2^2,3)$, the 
variety $X$ is given by one of the equations \eqref{eq:q=3o:eq3},
\eqref{eq:q=3o:eq2}, or \eqref{eq:q=3o:eq1} and $P$ is the point $(0,0,0,1,0)$. 
Then
the map $X \dashrightarrow \PP(1^2,2)$ from~\eqref{eq:q=3o:sl} is just the projection 
\[
(x_1,y_1,x_2,y_2, x_3) \longmapsto (x_1,y_1,x_2)
\]
and the equation of $\Delta_{\varphi}\subset \PP(1,1,2)$
is as follows:

case
\eqref{eq:q=3o:eq3}:
$(x_2+\phi_2'')(x_2^2\phi_2 +x_2\phi_4+\phi_6)-4(
x_2^2+x_2\phi_2'+\phi_4')^2=0$;

case
\eqref{eq:q=3o:eq2}:
$\phi_2''(x_2\phi_4+\phi_6)-4(x_2^2+\phi_4')^2=0$;

case
\eqref{eq:q=3o:eq1}:
$\phi_2''(x_2^3+x_2\phi_4+\phi_6)-4(x_2\phi_2'+\phi_4')^2=0$.
\end{remark}

\section{Nonrationality}
In this section we prove the following result.
\begin{theorem}
\label{thm:cb}
Let $\varphi: \tilde{X}\to \PP(1^2,2)$ be an extremal $\QQ$-conic bundle such that 
the singularities of $\tilde{X}$ along the fiber over $o:=(0,0,1)$ are 
moderate of indices $\le 2$, 
outside the fiber over $o:=(0,0,1)$
they are nodes or cusps, and the number of the singularities on $\tilde{X} \setminus \varphi^{-1}(o)$ is at most $4$.
Furthermore, assume that 
\begin{equation}
\label{eq:K-Delta1}
\Delta_{\varphi}\sim -2K_{\PP(1^2,2)}.
\end{equation} 
Then $\tilde{X}$ is not rational.
\end{theorem}

First of all note that by Proposition \ref{prop:cb-index2} the $\QQ$-conic 
bundle $\varphi$ over $o$ is of type \typeci{T}{2}
or \typec{ID_1^\vee} (see Examples \xref{ex:T} and \xref{ex:ID}). 
Then by Theorem~\ref{thm:SL-CB} there exists the following Sarkisov link
\begin{equation}
\label{eq:SL-CB-ns}
\vcenter{
\xymatrix@R=1em{
\check X\ar[d]_{p}\ar@{-->}[r]^{\chi}& Y\ar[d]^{\pi }
\\
\tilde X\ar[d]_{\varphi} & \FF_2\ar[ld]_{\mu}
\\
\PP(1^2,2)
} }
\end{equation}
where $p$ is an extremal blowup of a non-Gorenstein point, $\chi$ is a flip, $\mu$ is the minimal 
resolution of the point $o\in 
\PP(1^2,2)$, and
$\pi $ is a $\QQ$-conic bundle with 
at worst nodal or cusp singularities.
Let $\Sigma:=\mu^{-1}(o)$ be the 
negative section of $\FF_2$. Again by Theorem~\ref{thm:SL-CB} the variety $Y$ is smooth along $\pi^{-1}(\Sigma)$.
Moreover, if $\varphi$ over $o$ is of type \typeci{T}{2}, then $\Delta_{\varphi}$ does 
not pass through $o$ and
$\Delta_{\pi}\cap \Sigma=\varnothing$, and if $\varphi$ over $o$ is of type 
\typec{ID_1^\vee}, then $\Delta_{\varphi}\ni o$,
the pair $(S,\Delta_{\varphi})$ is lc at $o$, and
$\Delta_{\pi}\supset \Sigma$. In both cases by~\eqref{eq:K-Delta1} and~\eqref{eq:log-crepant} we have
\begin{equation}
\label{eq:K-Delta2}
\Delta_{\pi} \sim -2K_{\FF_2}. 
\end{equation} 

Assume that $\tilde X$ has no Gorenstein singularities. Then $Y$ is smooth and $\pi: 
Y\to \FF_2$ is a standard conic bundle
with $|2K_{\FF_2}+ \Delta_{\pi}|\neq \varnothing$. In this case $Y$ is not 
rational by \cite[Theorem~10.2]{Shokurov:Prym}.

From now on we assume that $X$ has $l$ Gorenstein singularities with $1\le l\le 4$.
The transformation $\tilde X\gets \check X \dashrightarrow Y$ does not affect 
the Gorenstein singular locus. 
Hence $Y$ has the same collection of Gorenstein singular points, say $P_1,\dots,P_l\in Y$. 
Let $s_i:= \pi(P_i)$, \ $i\in \{1,\dots,l\}$. By Lemma~\ref{lemma:new} the 
$\QQ$-conic bundle $\pi$ 
is of type \typeGor{n_i} (with $n_i\le 2$) near each fiber $\pi^{-1}(s_i)$.
In particular, $\Delta_{\pi}$ has multiplicity $2$ at $s_1,\dots,s_l$.
Moreover, $s_i\notin \Sigma$. 

Now, apply Theorem \ref{thm:SL-CB} in a neighborhood of each fiber $\pi^{-1}(s_i)$.
We obtain the following diagram extending \eqref{eq:SL-CB-ns}:
\begin{equation}
\label{eq:SL-CB-Go}
\vcenter{
\xymatrix@R=1em{
Y'\ar[d]_{\sigma}\ar@{-->}[r]^{\xi}& \tilde Y\ar[d]^{\tilde \pi}
\\
Y\ar[d]_{\pi} & \tilde S\ar[ld]_{\alpha}
\\
\FF_2
} }
\end{equation}
where $\sigma$ is the blowup of the points $P_1,\dots,P_l\in Y$, $\xi$ is a 
composition of flops, 
$\alpha$ is the blowup of the points $s_1,\dots,s_l\in \FF_2$,
and $\tilde \pi$ is a standard conic bundle.
Moreover, $\Delta_{\tilde \pi}$ is the proper transform of $\Delta_{\pi}$ and by~\eqref{eq:K-Delta2} and~\eqref{eq:log-crepant} we have
\begin{equation}
\label{eq:K-Delta3}
\Delta_{\tilde \pi}\sim -2K_{\tilde S}.
\end{equation} 
The next lemma shows that $\tilde S$ is a weak del Pezzo surface.

\begin{lemma}
\label{lemma:DP-DP}
In the above notation, $\tilde S$ is a weak del Pezzo surface of degree 
\[
d:=8-l\ge 4
\]
and 
the linear system $|-K_{\tilde S}|$ defines a birational crepant contraction
\begin{equation}
\label{eq:DP-DP}
\Phi_{|-K_{\tilde S}|}:\tilde S\longrightarrow \bar S\subset \PP^{d},
\end{equation} 
where 
$\bar S$ is an anticanonically embedded del Pezzo surface with Du Val singularities
and $K_{\bar S}^2=d$.
\end{lemma}

\begin{proof}
Assume that for some irreducible curve $\tilde C\subset \tilde S$ we have
$K_{\tilde S} \cdot \tilde C>0$. Then $\Delta_{\tilde{\pi}} \cdot \tilde C<0$, $\tilde C$ 
is a component of $ \Delta_{\tilde{\pi}}$, and $\tilde C^2<0$. Write $ \Delta_{\tilde{\pi}}=\tilde 
C+ \tilde{D}$,
where $ \tilde{D}$ is effective and does not contain $\tilde C$ as a component.
Then 
\[\textstyle
2\p(\tilde C)-2= K_{\tilde S} \cdot \tilde C+\tilde C^2 =\frac 12 \tilde 
C^2-\frac 12 \tilde{D} \cdot \tilde C<0.
\]
Hence, $\p(\tilde C)=0$, $\tilde C\simeq \PP^1$, and so 
\[
4= \tilde{D} \cdot \tilde C- \tilde C^2 > \tilde{D} \cdot \tilde C- \tilde C^2 + 
\Delta_{\tilde{\pi}} \cdot \tilde C =2 \tilde{D} \cdot \tilde C.
\]
On the other hand, since $\tilde C$ is a smooth rational curve, 
by Lemma~\ref{prop:cb} we have $ \tilde{D} \cdot \tilde C\ge 2$, a contradiction.

Thus the divisor $-K_{\tilde S}$ is nef. Since $(-K_{\tilde S})^2=d>0$, it is big.
Then the linear system $|-K_{\tilde S}|$ is base point free and defines a crepant contraction 
(see e.g. \cite[Theorem 8.3.2]{Dolgachev-ClassicalAlgGeom}).
\end{proof}

In order to apply Theorem~\ref{thm:sho} to $\tilde \pi: \tilde Y\to\tilde S$
we want $\Delta_{\tilde \pi}$ to 
satisfy the condition (S). However, this is not true in general. So, 
we need additional arguments.

\begin{lemma}
\label{lemma:surf}
Let $F$ be a smooth surface such that $-K_F$ is nef and $K_F^2\ge 4$,
and let $D\in |-2K_F|$ be a reduced divisor on $F$.
Suppose that there exists a decomposition 
\[
D=D'+D'', 
\]
where $D'$, $D''$ are
effective divisors with $D'\cdot D''=2$.
Then
either $-K_{F}\cdot D'=0$ or $-K_{F}\cdot D''=0$.
\end{lemma}

\begin{proof}
Assume that $-K_{F}\cdot D'>0$ and $-K_{F}\cdot D''>0$.
We have
\[
(-2 K_F)^2=D'^2+D''^2+4\ge 16,
\]
hence we may assume that $D''^2>0$.
Then by the Hodge index theorem 
\[
4=(D'\cdot D'')^2\ge D'^2\, D''^2.
\]
This implies that $D'^2\le 0$. Hence,
\[
2=D'\cdot D''=D'\cdot (-2K_F -D')=-2K_F\cdot D' -D'^2\ge 2.
\]
We obtain $K_F\cdot D'=-1$, $D'^2=0$.
On the other hand,
by the genus formula the number $K_F\cdot D'+D'^2$ must be even, a contradiction.
\end{proof}

\begin{corollary}
In the notation of~\eqref{eq:SL-CB-Go} suppose that the discriminant curve $\Delta_{\tilde{\pi}}$
admits a decomposition $ \Delta_{\tilde{\pi}}=\Delta_{\tilde{\pi}}'+\Delta_{\tilde{\pi}}''$
is a sum of effective divisors with $\Delta_{\tilde{\pi}}'\cdot\Delta_{\tilde{\pi}}''=2$.
Then
the anticanonical morphism \eqref{eq:DP-DP}
contracts either $ \Delta_{\tilde{\pi}}'$ or $ \Delta_{\tilde{\pi}}''$.
\end{corollary}

\begin{proof}[Proof of Theorem~\xref{thm:cb}]
We use the notation of \eqref{eq:SL-CB-Go}.
By Lemma~\ref{lemma:DP-DP}\ $\tilde S$ is a weak del Pezzo surface of degree $d\ge 4$,
the linear system $|-K_{\tilde S}|$ is base point free and defines a crepant contraction 
$\psi: \tilde S\to \bar S$, where 
$\bar S$ is a del Pezzo surface with Du Val singularities. Let $\bar \Delta:=\psi_*\Delta_{\tilde{\pi}}$. 
By~\eqref{eq:K-Delta3} we have
\begin{equation*}
\bar \Delta\sim -2K_{\bar S}. 
\end{equation*} 
For any $\psi$-exceptional curve $\tilde E$ we have $K_{\tilde S}\cdot \tilde E= \Delta_{\tilde{\pi}}\cdot \tilde E=0$.
Hence, $\psi$ is log crepant with respect to $K_{\tilde S}+ \Delta_{\tilde{\pi}}$,
the pair $(\bar S,\bar \Delta)$ is lc, and the curve $\bar \Delta$ has at worst nodal
singularities.
By Lemma~\ref{lemma:surf} the curve $\bar \Delta$ satisfies the condition (S) of Theorem~\ref{thm:sho}.
Let $\tilde \tau: \hat \Delta_{\tilde{\pi}}\to 
\Delta_{\tilde{\pi}}$ be the double cover associated to $\tilde{\pi}$.
Then there is a similar double cover $\bar \tau: \hat \Delta\to \bar \Delta$ 
that coincides with $\tilde \tau$ over
$\Delta_{\pi}\setminus \Exc(\psi)$ and by Remark~\ref{rem:Prym} we have a natural isomorphism of Prym varieties
\[
\Pr( \hat \Delta_{\tilde{\pi}}/\Delta_{\tilde{\pi}} ) \simeq \Pr(\hat 
\Delta/\bar \Delta).
\]
By Lemma~\ref{lemma:DP-curves-2} below
the curve $\bar \Delta$ is neither hyperelliptic nor trigonal nor quasi-trigonal, and also 
$\bar \Delta$ is not a plane quintic.
Then by Theorem \ref{thm:sho} the Prym variety $ \Pr(\hat
\Delta/\bar \Delta)$ is not a sum of Jacobians of curves, hence $\tilde Y$ is not rational by Theorem~\ref{thm:J-Pr}.
\end{proof}

\begin{lemma}
\label{lemma:DP-curves-2}
Let $F$ be a del Pezzo surface with Du Val singularities and let $\Xi\in 
|-2K_F|$ be a normal crossing curve.
If $K_F^2\ge 3$, then the curve $\Xi$ is not hyperelliptic.
If $K_F^2\ge 4$, then the curve $\Xi$ is neither trigonal
nor quasi-trigonal and it is not a plane quintic.
\end{lemma}

\begin{proof}
Let $d:=K_F^2$.
Since $d\ge 3$, the linear system $|-K_F|$ is very ample and defines 
an embedding $\Phi_{|-K_F|}:F\hookrightarrow \PP^{d}$ (see e.g. 
\cite{Hidaka-Watanabe}). 
By the adjunction formula $K_{\Xi}=(K_F+\Xi)_{\Xi}=-K_F|_{\Xi}$.
By the Kawamata-Viehweg vanishing the homomorphism 
\[
H^0(F,\OOO_F(-K_F))\longrightarrow H^0(\Xi,\OOO(K_\Xi))
\]
is surjective,
hence the restriction $\Phi_{|-K_F|}|_\Xi:\Xi \hookrightarrow \PP^d$ is 
the canonical map and $\p(\Xi)= d+1$. In particular, the curve $\Xi$ is 
not hyperelliptic.

Now assume that $\Xi$ is either trigonal or quasi-trigonal.
Then by the geometric version of the Riemann-Roch theorem the 
canonical model $\Xi \subset \PP^d=\PP^{\p(\Xi)-1}$ has $3$-secant lines.
On the other hand, it is known that the anticanonical model $F\subset \PP^d$ of 
a Du Val surface is projectively normal 
and is an intersection of quadrics if $d\ge 4$ \cite{Hidaka-Watanabe}. Hence
the homomorphism 
\[
H^0(\PP^d,\OOO_{\PP^d}(2))\longrightarrow H^0(F,\OOO(-2K_\Xi))
\]
is 
also surjective and so the curve $\Xi\subset \PP^d$ is an intersection 
of quadrics as well. In this case $\Xi$ cannot have $3$-secant lines. 
Finally, assume that $\Xi$ is a plane quintic, i.e. it admits an embedding $\Xi\subset \PP^2$ as a 
curve of degree $5$. Then $\p(\Xi)=6$ and the restriction 
of the Veronese embedding $\PP^2 \hookrightarrow \PP^5$ to $\Xi$ is also a canonical map.
Therefore, the canonical image of $\Xi$ is contained in a Veronese surface: $\Xi\subset V\subset \PP^5$.
On the other hand,
by the above, the curve $\Xi\subset \PP^5$ is an intersection of quadrics.
Therefore, $\Xi$ is contained in $V\cap Q$ for some quadric $Q$.
But then $\deg \Xi\le \deg (V\cap Q)=8$, a contradiction.
\end{proof}

\begin{proof}[Proof of Theorem~\xref{thm:main}]
If $X$ has a non-Gorenstein singularity which is not moderate, then $X$ is rational by Corollary~\ref{cor:n-mod-rat}.
Assume that all non-Gorenstein singularities are moderate. In this situation, by Theorem~\ref{thm:q=3o:main} there exists a Sarkisov link of the form
\eqref{eq:q=3o:sl}. Then we apply Theorem~\ref{thm:cb} to the $\QQ$-conic bundle $\varphi :\tilde X\to \PP(1^2,2)$.
\end{proof}

Let $\mathscr{X}_6$ be the family of all hypersurfaces
$X_6\subset\PP(1^2,2^2,3)$. Theorem~\xref{thm:main} shows, in particular, that a general member of $\mathscr{X}_6$
is non-rational.
The following proposition presents a special subfamily in $\mathscr{X}_6$ with non-rational members.
As was noticed by I. Cheltsov, this gives another proof of non-rationality of
a \emph{very general} member of $\mathscr{X}_6$ \cite[Theorem~1.8]{Kollar-1996-RC}.

Recall that a point $R$ on a cubic hypersurface $Y=Y_3\subset \PP^4$ is an \textit{Eckardt point} 
if the intersection of $Y$ with the tangent space $T_{R,Y}$ is a cone over an elliptic curve.

\begin{proposition}
\label{prop:Eckardt}
For a smooth point $P\in \PP(1^2,2^2,3)$, let $\mathscr{X}_6^P\subset \mathscr{X}_6$ be the subfamily consisting of all the members 
having at $P$ a singularity that is worse than of type \type{cA}. 
Then the following assertions hold.

\begin{enumerate}
\item
\label{prop:Eckardt0}
The singular locus of a general variety $X_6\in \mathscr{X}_6^P$ consists of 
three cyclic quotient singularities of type $\frac12(1,1,1)$ and a Gorenstein singularity of type \types{cD}{4}
at $P$.

\item
\label{prop:Eckardt1}
A general member $X_6\in \mathscr{X}_6^P$ is birational to a smooth 
cubic hypersurface $Y_3\subset \PP^4$ having an Eckardt point.

\item 
\label{prop:Eckardt2}
Conversely, any smooth cubic hypersurface $Y_3\subset \PP^4$ with an Eckardt point
is birational to a member of $\mathscr{X}_6^P$.
\end{enumerate}
\end{proposition}

\begin{proof}
For \ref{prop:Eckardt0} and \ref{prop:Eckardt1} we may assume that $P=(1,0,0,0,0)$. Then the equation of $X_6$ can be written in the form
\begin{equation}
\label{eq:cubic3}
x_3^2+\phi_6(x_2,y_2)+y_1\gamma(x_1,y_1,x_2,y_2)+ x_1^2\delta(x_1,x_2,y_2) =0,
\end{equation}
where $\phi_6$, $\gamma$, and $\delta$ are quasi-homogeneous polynomials of degree $6$, $5$, and $4$, respectively.
If the singularity of $X$ at $P$ is worse than \type{cA}, then $\delta(1,x_2,y_2)=0$ and $\gamma(1,y_1,x_2,y_2)$ contains no terms of (usual) degree $\le 1$. Then \eqref{eq:cubic3} has the 
form 
\begin{multline}
\label{eq:cubic2}
x_3^2+\phi_6(x_2,y_2)
+x_2^2y_1\phi_1(x_1,y_1)+x_2y_2y_1\phi_1'(x_1,y_1)+y_2^2y_1\phi_1''(x_1,y_1)+
\\
+
x_2y_1^2\phi_2(x_1,y_1)+y_2y_1^2\phi_2'(x_1,y_1)
+c_3x_1^3y_1^3+c_2x_1^2y_1^4+c_1x_1y_1^5+c_0y_1^6=0,
\end{multline}
where $\phi_d$, $\phi_d'$, $\phi_d''$ are homogeneous polynomials of degree $d$ and 
$c_i$ are constants.
If $\phi_d$, $\phi_d'$, $\phi_d''$, $c_i$ are general, then by Bertini's theorem
the singular locus of $X_6$ consists of 
three points of type $\frac12(1,1,1)$ and a Gorenstein singularity of type \types{cD}{4}
at $P$.

In the affine chart $U:=\{y_1\neq 0\}\simeq \mathbb{A}^4$ the variety $X_6\cap U$
is given by the cubic equation
\begin{equation}
\label{eq:Eckardt}
x_3^2+ \phi(x_2,y_2, x_1)=0,
\end{equation} 
where $\phi$ a general polynomial of degree $3$.
It is easy to see that the closure of $X_6\cap U$ in $\PP^4$
is a smooth cubic hypersurface $Y_3\subset \PP^4$ with an Eckardt point.
This proves \ref{prop:Eckardt0} and \ref{prop:Eckardt1}.

\ref{prop:Eckardt2}
Conversely, 
consider a cubic $Y=Y_3\subset \PP^4$ with an Eckardt point. In a suitable affine coordinates it can be given by the equation \eqref{eq:Eckardt}. Assign to variables $x_1,x_2,y_2,x_3$ weights 
$1,2,2,3$, respectively and write \eqref{eq:Eckardt} in an expanded form:
\[
x_3^2+\phi_6(x_2,y_2)
+x_2^2\phi_1(x_1)+x_2y_2\phi_1'(x_1)+y_2^2\phi_1''(x_1)+
x_2\phi_2(x_1)+y_2\phi_2'(x_1)
+c_3x_1^3+c_2x_1^2+c_1x_1+c_0=0.
\]
Now, introduce a new variable $y_1$ of weight $1$ and homogenize the last equation with respect to 
prescribed weights. We obtain \eqref{eq:cubic2}.
\end{proof}

\begin{remark}
Note however that a general member $X_6\in \mathscr{X}_6$ is not birational to a cubic threefold.
Indeed, in this case the $\QQ$-conic bundle $\varphi: \tilde X\to \PP(1^2,2)$ from Theorem~\ref{thm:q=3o:main} is of type \typeci{T}{2} over the singular point of $\PP(1^2,2)$ and the discriminant 
curve $\Delta_{\varphi}$ is a smooth curve of genus $9$. Then the variety $Y$ in the diagram~\eqref{eq:SL-CB-ns}
is smooth and its Intermediate Jacobian $\J(Y)$ is isomorphic to the Prym variety 
$\Pr(\hat \Delta_{\varphi}/\Delta_{\varphi})$ which is a simple $8$-dimensional principally polarized 
abelian variety \cite{Mumford1974}, \cite{Shokurov:Prym}. Hence $\J(Y)$ cannot contain the Intermediate Jacobian of a cubic threefold as a direct summand. 
\end{remark}


\end{document}